\documentclass{article}
\usepackage{amssymb}
\usepackage{amsthm}
\usepackage{amsmath}
\usepackage{graphicx}
\usepackage{graphicx}
\usepackage{float}
\usepackage{subfigure}
\usepackage{indentfirst}
\usepackage{amsthm}
\newtheorem{theorem}{Theorem}[section]
\newtheorem{lemma}[theorem]{Lemma}
\newtheorem{problem}[theorem]{Problem}
\newtheorem{definition}[theorem]{Definition}
\newtheorem{corollary}[theorem]{Corollary}

\theoremstyle{plain}
\newtheorem{claim}{Claim}
\theoremstyle{definition}

\usepackage{caption}
\usepackage{amssymb}
\usepackage{hyperref}
\usepackage{geometry}
\usepackage{pifont}
\usepackage{tikz}
\usepackage{url}
\usepackage{amsmath,amssymb,amsfonts}
\usepackage{color}
\usepackage{setspace}
\usepackage{cite}
\usepackage[utf8]{inputenc}
\usepackage[misc]{ifsym}
\title{\huge On  minimally 1-tough  $(K_1\cup P_4)$-free graphs\thanks{{Supported by the National Natural Science Foundation of China (No. 12271311) and the Natural Science Foundation of Shandong Province (No. ZR202102250232).}}}
\author{Shiyu Cao, Jing Chen, Wei Zheng\footnote{\scriptsize Corresponding author.
E-mail addresses: zhengweimath@163.com (W. Zheng).},\\
\small School of Mathematics and Statistics, Shandong Normal University, Jinan 250358, Shandong, China}
\date{}

\usepackage[numbers,sort&compress]{natbib}
\hypersetup{colorlinks=true,linkcolor=black,citecolor=black,urlcolor=black}

\begin{document}

	\maketitle
	\vspace{-2.5cm}
	\vskip20pt \baselineskip12pt \vskip40pt\baselineskip16pt

	\noindent {\bf Abstract.}
	\noindent A graph $G$ is minimally $t$-tough if the toughness of $G$ is $t$ and the deletion of any edge from $G$ decreases its
	toughness, where $t$ is a positive real number.
	It is conjectured that every $(K_1\cup P_4)$-free 1-tough graph  is hamiltonian. In this paper, we characterize the structure of minimally 1-tough   $(K_1\cup P_4)$-free graphs, and thus show that the above conjecture is true for minimally 1-tough graphs. Furthermore, it is also proved that the Kriesell's conjecture which states that each minimally 1-tough graph has  a vertex of degree 2  holds for minimally 1-tough   $(K_1\cup P_4)$-free graphs.
	\rm
	
	\vskip10pt
	
	\noindent {\bf Keywords.}
	\noindent toughness; hamiltonian graph; minimally $t$-tough graphs;   $(K_1\cup P_4)$-free.
	
	\section{Introduction}
	
	All	graphs considered in this paper are simple and finite.
	Let $\omega(G)$ denote the number of   components of a graph
	$G$.
	Let $P_n$ be a path on $n$ vertices, and let $C_n$ be a cycle on $n$ vertices.
	For a vertex subset $S$ of $G$, the notation $G-S$ represents the graph obtained by deleting all the vertices in $S$ from $G$, along with the edges incident to those vertices.
	We say $S$ is a  vertex cut of $G$ if the subgraph $G-S$ is disconnected.
	In 1973, Chvátal  \cite{chvatal1973tough} introduced the notion of toughness.

	\begin{definition}
		\emph{Let $t$ be a non-negative real number. A graph $G$ is  \sl $t$-tough \rm if $|S|\geq  t\cdot \omega(G-S)$ for any
			vertex cut $S$ of $G$. The  \sl toughness \rm   of $G$, denoted by $\tau(G)$, is the maximum number $t$ such that $G$ is $t$-tough if $G$ is non-complete, and let $\tau(G)=+ \infty$ if $G$ is a complete graph.}
	\end{definition}
	
	A graph $G$ is \sl hamiltonian \rm if it  has a cycle that contains all the vertices of $G$.
	The toughness of a graph is closely related to its hamiltonicity.
	In the process of studying the relationship between  toughness and hamiltonicity of a graph, Chvátal \cite{chvatal1973tough} proposed a  conjecture that there exists a constant $t_0$ such that each $t_0$-tough graph is hamiltonian.
	Bauer et al.~\cite{BauerNot every 2-tough} by proving the existence of a $(\frac{9}{4}-\epsilon)$-tough non-hamiltonian graph for every $\epsilon>0$, showed that if the above Chvátal's conjecture is true, then $t_0\geq \frac{9}{4}$.
	Chvátal's conjecture is still open.

	We use $H_1\cup H_2$ represents the union of two vertex-disjoint graphs $H_1$ and $H_2$.
	For a given graph $H$, we say a graph $G$ is \sl $H$-free \rm  if it does not contain an induced subgraph that is isomorphic to $H$, and  $H$ is called a \sl forbidden subgraph \rm   for $G$.
	It has been demonstrated that the above Chvátal's conjecture  holds true for some graphs with forbidden subgraph conditions.
	It is clear that if a graph is hamiltonian, then it is 1-tough.
	However, this conclusion is not true in reverse.
	It is interesting to consider which forbidden subgraphs $H$ guarantee that each 1-tough $H$-free graph is hamiltonian.
	Nikoghosyan \cite{Nikoghosyan} showed that every 1-tough $H$-free graph is hamiltonian if $H$ is $K_1\cup P_2$ or $K_1\cup P_3$, and  also  conjectured that each 1-tough $P_4$-free graph is hamiltonian, which has been proven by  Jung in \cite{jung}.
	Li et al.~\cite{LibinlP4K1} characterized all graphs $H$ such that every 1-tough $H$-free graph is hamiltonian and got the following two theorems.
	
	\begin{theorem}[\cite{LibinlP4K1}]
		\sl
		Let $R$ be an induced subgraph of $P_4$, $K_1\cup P_3$ or $2K_1\cup K_2$. Then every $R$-free 1-tough graph on at least three vertices is hamiltonian.
		\rm
	\end{theorem}
	
	\begin{theorem}[\cite{LibinlP4K1}]
		\sl
		Let $R$ be a graph on at least three vertices. If every $R$-free 1-tough graph on at least three vertices is hamiltonian, then $R$ is an induced subgraph of $K_1\cup P_4$.
		\rm
	\end{theorem}

	It is evident that every induced subgraph of $K_1\cup P_4$ is either the graph $K_1\cup P_4$ itself or an induced subgraph of $P_4$, $K_1\cup P_3$, or $2K_1\cup K_2$.
	Combining with  the above two theorems, it can be seen that $K_1\cup P_4$ is the only remaining case to  characterize all graphs $H$ such that every $H$-free 1-tough graph is hamiltonian.
	Li et al.~\cite{LibinlP4K1} raised the following problem, which
	was also proposed by Nikoghosyan in \cite{Nikoghosyan}.

	\begin{problem}[\cite{LibinlP4K1}]
		\sl
		\label{problem}
		Is every $(K_1\cup P_4)$-free 1-tough graph on at least three vertices hamiltonian?
		\rm
	\end{problem}
	
	This problem seems difficult to solve up to now.
	It is also hard even to consider whether each  $(K_1\cup P_4)$-free $t$-tough graph with $t>1$ is  hamiltonian.
	However, upon imposing an additional condition requiring these graphs to be triangle-free,
	Zheng et al.~\cite{1-tough triangle-free} proved that every 1-tough $\{\triangle, K_1\cup P_4\}$-free graph on at least three vertices is hamiltonian. Here, the symbol $\triangle$ denote the triangle.
	In this paper, we give a positive answer to Problem \ref{problem} for graphs with toughness  exactly 1 and the deletion of each
	edge will decrease its toughness.
	In fact, the graph mentioned above is called a \sl minimally 1-tough graph\rm,
	which was introduced by Broersma et al.~in \cite{broersma1999various}.
	
	\begin{definition}
		
		\emph{Let $t$ be a positive real number. A graph $G$ is  a  \sl minimally $t$-tough graph \rm if $\tau(G)=t$ and $\tau(G-e) < t$ for each edge $e$ of $G$.}
	\end{definition}

	In this paper, we completely characterize the structure of minimally 1-tough   $(K_1\cup P_4)$-free graphs.

	\begin{theorem}
		\sl
		\label{Th1}
		Let $G$ be a minimally 1-tough   $(K_1\cup P_4)$-free graph. Then $G$ is $C_4$, $C_5$ or $C_6$.
		\rm
	\end{theorem}
	
	Using Theorem \ref{Th1}, it can be obtained that every minimally 1-tough   $(K_1\cup P_4)$-free graph is  hamiltonian, and indicates that Problem \ref{problem} holds for minimally 1-tough   $(K_1\cup P_4)$-free graphs.
	
	\begin{corollary}
		\sl
		\label{th2}
		Every minimally 1-tough   $(K_1\cup P_4)$-free graph on at least three vertices is hamiltonian.
		\rm
	\end{corollary}
	
	By the definition of toughness,
	it is obvious that each non-complete $t$-tough graph is $\lceil 2t \rceil$-connected, and thus the minimum degree of which is at least $\lceil 2t \rceil$.
	In 2003, Kriesell \cite{kaiser2003problems} stated a conjecture that each minimally 1-tough graph has a vertex of degree 2. According to 	Theorem \ref{Th1}, we also confirm
	Kriesell's conjecture  for minimally 1-tough   $(K_1\cup P_4)$-free graphs.

	\begin{corollary}
		\sl
		\label{th3}
		The minimum degree of each minimally 1-tough $(K_1\cup P_4)$-free graph is 2.
		\rm
	\end{corollary}

	The paper is organized as follows. In Section 2, we give some definitions and technical lemmas that will
	be used in the proof of  Theorem \ref{Th1} in Section 3.

	\section{Preliminaries  and notations}
	
	Let $G=(V(G), E(G))$ be a simple graph.
	For a vertex $v$ of $G$, let $N_G(v)$ denote the neighbor set of $v$ in $G$ and denote	$N_G[v]=N_G(v) \cup \{v\}$.
	The distance between two vertices $u$ and $v$ in $G$, denoted by $d_G(u,v)$, is the length of a shortest $(u,v)$-path in $G$.
	For a  vertex subset $S$ of $G$, 	let $G[S]$ denote the induced subgraph of $S$ in $G$.
	If $v$ is adjacent to at least one vertex of $S$,  we say $v$ is adjacent to $S$.
	For two vertex subsets $S$ and $T$ of $G$, let
	$N_S(T)$ represent the set of neighbors of all vertices of $T$ in $S$ and $S-T$ represent the set obtained from $S$ by removing all the members of $T$.
	In particular, if $T$ contains exactly one vertex $v$, we write $N_S(T)$ as $N_S(v)$, and $S-T$ as $S-\{v\}$.
	A complete bipartite graph is defined as a simple bipartite graph with bipartition $(X,Y)$ such that each
	vertex of $X$ is adjacent to all the vertices of $Y$, here $X$ or $Y$ may be an empty set.
	Definitions and notations not mentioned here can be  referred to
	\cite{bondy2008graph}.
	
	The following lemma is given by Katona et al.~\cite{katona2018properties}, which characterizes a basic property of minimally 1-tough graphs.

	\begin{lemma}[\cite{katona2018properties}]
		\sl
		\label{1-tough}
		If $G$ is a minimally 1-tough graph, then
		for each edge $e$ of $G$, 
		 there exists a vertex subset $S$ of $G$ such that
		\begin{center}
			$\omega(G - S) = |S|$ and $\omega((G - e) - S) = |S| + 1.$
		\end{center}
		\rm
	\end{lemma}
	
	For the vertex set  $S$ that satisfies the above lemma for the edge $e$, the vertex in $S$ also has some nice property when the number of $S$ satisfies certain conditions.

	\begin{lemma}
		\label{two components}
		\sl
		Let $G$ be a minimally 1-tough graph, and let $S$ be a vertex set which satisfies Lemma \ref{1-tough} for an edge $e$.
		If $|S|\geq 2$, then  each vertex of $S$ is adjacent to at least two components of $G-S$.
		\rm
	\end{lemma}
	
	\begin{proof}[\bf{Proof.}]
		Since $|S|\geq 2$, we have   $\omega(G-S)=|S|\geq 2$.
		Suppose to the contrary that there exists a vertex  $s_i$ of $S$ such that $s_i$ is adjacent to at most one 
	component of $G-S$. 
	Take $S'=S-\{s_i\}$. Then either $\omega(G-S')=\omega(G-S)$ or $\omega(G-S')=\omega(G-S)+1$, which means that 
	$\omega(G-S')\geq \omega(G-S)$. Hence, $S'$ is a vertex cut of $G$.  
	Since $\tau(G)=1$ and  $\omega(G - S) = |S|$, we have  $$1=\tau(G)\leq \frac{|S'|}{\omega(G-S')}\leq\frac{|S|-1}{\omega(G-S)}=\frac{\omega(G-S)-1}{\omega(G-S)}<1,$$
	a contradiction. Therefore, each vertex  of $S$ is adjacent to at least two components of $G-S$.
	\end{proof}

	The following result characterizes the degree of vertices of each triangle  in  minimally 1-tough graphs, which is useful to prove the main result.

	\begin{lemma}[\cite{katona2018properties}]
		\sl
		\label{triangle}
		The degree of each vertex of any triangle in a minimally 1-tough graph is at least 3.
		\rm
	\end{lemma}
	
	Note that the toughness of a  hamiltonian graph is at least 1,
	and the deletion of an edge which is not contained by the Hamilton cycle will not reduce the toughness below 1.
	Thus, the following result holds obviously.

	\begin{lemma}
		\sl
		\label{hamiltonian}
		Let $G$ be a hamiltonian graph. If $G$ is not a cycle, then $G$ is not a  minimally 1-tough graph.
		\rm
	\end{lemma}

	\section{Proof of main result}

	Let $G$ be a graph.
	For an arbitrary edge $e$ with endpoints $x$ and $y$ in $G$, let  
	\begin{center}
			$g(xy)=|N_G(x) \cap N_G(y)|$ and $f(xy)=|N_G(x) \cup N_G(y)|$.
	\end{center}
Let  $S$ be a vertex set  which satisfies Lemma \ref{1-tough} for the edge $e$. Denote $C(e)$ as the component of $G-S$ which contains the edge $e$ and $D(e)$ as the union of the components of $G-S-C(e)$.
	In this section, we present the proof of Theorem \ref{Th1}, stating that  if  $G$ is a minimally 1-tough   $(K_1\cup P_4)$-free graph, then  $G$ is $C_4$, $C_5$ or $C_6$.

	\begin{proof} [\bf{Proof of Theorem \ref{Th1}.}]
		Let $e=uv$ be an edge of $G$
		such that
		
		(i) $g(uv)=$ max$\{g(xy):xy \in E(G)\}$;
		
		(ii) under the condition (i),
		$f(uv)=$ max$\{f(xy):xy \in E(G)\}$.
		
		\noindent	Let $S$ be a vertex subset of $G$ which satisfies Lemma \ref{1-tough} for the edge $e$. Then   
		\begin{center}
			$\omega(G-S)=|S|$ and $\omega((G-e)-S)=|S|+1.$
		\end{center}
		Since  $e$ is a bridge of the component $C(e)$ in $G-S$, we have $N_G(u) \cap N_G(v) \subseteq S$. Thus, $g(uv) \leq |S|$.
		Let $G_u$ and $G_v$ be the two components of $C(e)-e$ which contain $u$ and $v$, respectively. Denote the vertex sets of $G_u$ and $G_v$ by $L_u$ and  $L_v$, respectively.
		More clearly, let
		$$L_u=\{u\}\cup L_1\cup\cdots\cup L_a,$$ $$L_v=\{v\}\cup R_1\cup\cdots\cup R_b,$$ where $$L_i=\{x | d_{G_u}(x,u)=i, x\in L_u\},~~0\leq i\leq a,$$
		$$R_j=\{y | d_{G_v}(y,v)=j,y\in L_v\},~~ 0\leq j\leq b.$$
		Here, the amount $a$ denotes the maximum distance between the vertex $u$ and any other vertex in $G_u$, and $b$ denotes the maximum distance between the vertex $v$ and any other vertex in $G_v$.
		Clearly, $L_0=\{u\}$, $R_0=\{v\}$.

		\begin{claim}
			\label{connected}
			\rm
			If $L_u-\{u\}\neq \emptyset$, then the induced subgraph $G[L_u-\{u\}]$ is connected, and if $L_v-\{v\}\neq \emptyset$, then  the induced subgraph  $G[L_v-\{v\}]$ is connected.
		\end{claim}
		\begin{proof}[\rm{Proof.}]
			Suppose to the contrary that $G[L_u-\{u\}]$ is not connected.
			It follows from  $L_u-\{u\}\neq \emptyset$ that
			$S\cup \{u\}$ is a vertex cut in $G$, and  $$\omega(G-(S\cup \{u\}))=\omega(G_u-\{u\})+1+|D(e)|\geq 2+1+|D(e)|=
			3+\omega(G-S)-1=|S|+2.$$
			However, since $\tau(G)=1$, by the definition of toughness, it follows that
			$$\omega(G-(S\cup \{u\}))\leq |S\cup \{u\}|=|S|+1.$$
			This leads to a contradiction. Thus, the subgraph $G[L_u-\{u\}]$ is  connected. Similarly, the subgraph $G[L_v-\{v\}]$ is also connected if $L_v-\{v\}\neq \emptyset$.
		\end{proof}
		We divide  the following discussion into two cases.

		\vskip3pt\noindent{\bf Case 1.} $S$ is not a vertex cut of $G$.
		
		In this case, $|S|=\omega(G-S)=1$. Then $g(uv)=0$ or 1.
		Let $S=\{s\}$.
		Since $G$ is $(K_1\cup P_4)$-free, the length of each induced
		path in $G$ is at most 4.
		Note that $e$ is a bridge of $C(e)$. By the definition of $L_i$ and $R_j$,
		there exists an induced path with one end in
		$L_a$ and another end in $R_b$ in the component $C(e)$. Therefore,  $a+b+1\leq 4$, and so $a+b\leq 3$.

		If $g(uv)=1$, then $us,vs\in E(G)$, and so $uvsu$ is a triangle of $G$. By Lemma \ref{triangle}, we can obtain    $L_1\neq \emptyset$ and $R_1\neq \emptyset$. Since $g(uv)=1$, $s$ is neither adjacent to $L_1$ nor to $R_1$.
		Because $u$ and $v$ are not cut vertices of $G$, so $s$ is adjacent to both $L_u-(\{u\}\cup L_1)$ and $L_v-(\{v\}\cup R_1)$.
		Thus,
		$L_2\neq \emptyset$ and $R_2\neq \emptyset$. Then $a+b\geq 4$, contradicting   that $a+b\leq 3$.
		
		In the following discussion, assume that $g(uv)=0$.
		We first characterize the structures of $G_u$ and $G_v$ in the claims below.

		\begin{claim}
			\label{independent}
			\rm
			If $R_j\neq \emptyset$, then $R_j$ is an independent set of $G$, where $1\leq j\leq b\leq 3$, and  if $L_i\neq \emptyset$, then $L_i$ is an independent set of $G$, where $1\leq i\leq a\leq 3$.
		\end{claim}
		\begin{proof}[\rm{Proof.}]
			Since $g(uv)=0$, i.e., $G$ is triangle-free,
			it is obvious that $R_1$ is an independent set.
			We now show that $R_2$ is an independent set if $R_2\neq \emptyset$.
			Suppose to the contrary that    $R_2$ is not an independent set.
		Without loss of generality, let $y_1$ and $y_2$ be two adjacent  vertices in $R_2$. Clearly, $N_G(y_1)\cap N_G(y_2)=\emptyset$. Let $y_1'$ and $y_2'$ be the neighbors of $y_1$ and $y_2$ in $R_1$, respectively.
			Since $R_1$ is an independent set, $y_1'y_2'\notin E(G)$. Then the subgraph induced by $\{u,y_1',y_1,y_2,y_2'\}$ is a  $K_1\cup P_4$ of  $G$, a contradiction.
			By 	similar discussion, we can obtain   $R_3$ is also an independent set if it exists.  Therefore,  $R_j$ is an independent set if
			$R_j\neq \emptyset$ for $1\leq j\leq 3$. Similarly, $L_i$ is also an independent set if $L_i\neq \emptyset$ for $1\leq i\leq 3$.
		\end{proof}

		\begin{claim}
			\label{neighbirs}
			\rm
			If $R_2\neq \emptyset$, then each vertex of $R_1$ has neighbors in $R_2$, and if $L_2\neq \emptyset$, then each vertex of $L_1$ has neighbors in $L_2$.
		\end{claim}
		
		\begin{proof}[\rm{Proof.}]
			If $|R_1|=1$, then the conclusion holds obviously. Suppose that $|R_1|\geq 2$ and to the contrary that there exists a vertex $y\in R_1$ such that $N_{R_2}(y)=\emptyset$. By Claim \ref{independent}, $R_1$ is an independent set.
			Then $y$ is not adjacent to any vertex of $R_1-\{y\}$. Thus, the subgraph $G[L_v-\{v\}]$ is not connected, which contradicts Claim \ref{connected}.
			 Therefore, if $R_2\neq \emptyset$, then each vertex of $R_1$ has neighbors in $R_2$.  By similar discussion, if $L_2\neq \emptyset$, then each vertex of $L_1$ has neighbors in $L_2$.
		\end{proof}

		\begin{claim}
			\label{complete bipartite}
			\rm
			The induced subgraphs $G[R_j\cup R_{j+1}]$ and $G[L_i\cup L_{i+1}]$ are  complete bipartite graphs, where $j=1,2$ and $i=1,2$.

		\end{claim}
		\begin{proof}[\rm{Proof.}]
			By Claim \ref{independent}, $R_j$ and $R_{j+1}$ are independent sets if $R_j\neq \emptyset$ and $R_{j+1}\neq \emptyset$.
			Clearly, if $R_{j+1}=\emptyset$, then $G[R_j\cup R_{j+1}]$ is a complete bipartite graph.
			If $R_{j+1}\neq \emptyset$, then we only need to show that each vertex of $R_j$ is adjacent to all the vertices of $R_{j+1}$.
			For  $j=1$, suppose to the contrary that there exists a vertex $y_1\in R_1$ and $y_2\in R_2$ such that $y_1y_2\notin E(G)$. By Claim \ref{neighbirs}, let $y_1'$ be a neighbor of $y_1$ in $R_2$.
			Then the subgraph induced by $\{y_2,u,v,y_1,y_1'\}$ is a
			$K_1\cup P_4$ of  $G$, a contradiction.  Therefore, the subgraph $G[R_1\cup R_2]$ is a complete bipartite graph.
			For $j=2$,
			if  there exists a vertex $y\in R_2$ and $y'\in R_3$ such that $yy'\notin E(G)$.
			For any vertex $y''\in R_1$, 
			the subgraph induced by $\{y',u,v,y'',y\}$ 
			is a  $K_1\cup P_4$ of  $G$, a contradiction. Therefore,  the subgraph $G[R_2\cup R_3]$ is a complete bipartite graph if $R_3\neq \emptyset$. Thus, the induced subgraph $G[R_j\cup R_{j+1}]$ is a complete bipartite graph if  $R_{j+1}\neq \emptyset$ for each $j=1,2$. By similar discussion, we can obtain  $G[L_i\cup L_{i+1}]$ is a complete bipartite graph if  $L_{i+1}\neq \emptyset$ for each $i=1,2$.
		\end{proof}

		\begin{claim}
			\rm
			\label{path}
			If $vs\notin E(G)$, then the induced subgraph $G[L_v]$ is a path and $s$ is adjacent to $R_b$, and if $us\notin E(G)$, then the induced subgraph $G[L_u]$ is a path and $s$ is adjacent to $L_a$.
		\end{claim}
		
		\begin{proof}[\rm{Proof.}]
			Since $vs\notin E(G)$ and $e$ is not a bridge of $G$, we have  $R_1\neq\emptyset$ and $s$ is adjacent to $L_v-\{v\}$.
			Since $a+b\leq 3$, we only need to discuss the cases that $b=1,2$ or 3.

			If  $b=1$, then $L_v=\{v\}\cup R_1$ and $s$ is adjacent to $R_1$.
			By Claim \ref{independent}, $R_1$ is an independent set.
			Applying  Claim \ref{connected}, $G[R_1]$ is a connected subgraph. Thus, $|R_1|=1$. Therefore, in this case, $G[L_v]$ is a path of length 1 and $s$ is adjacent to $R_1$.

			Since $e$ is not a bridge of $G$, $s$ is adjacent to at least one vertex of $L_u$.
			Then $G[L_u\cup \{s\}]$ is connected. Let $P$ be a path from
			$s$   to $u$  in $G[L_u\cup \{s\}]$.
			In the following discussion, for convenience, denote the vertex sets $R_1=\{w_1,w_2,\dots, w_p\}$, $R_2=\{y_1,y_2,\dots,y_q\}$ and $R_3=\{z_1,z_2,\dots,z_l\}$.

			If $b=2$, then $L_v=\{v\}\cup R_1\cup R_2$.
			By Claim \ref{complete bipartite},
		 	$G[R_1\cup R_2]$ is a complete bipartite graph. Then
			$s$ is adjacent to exactly one of $R_1$ and $R_2$ since  $g(uv)=0$.
			If  $s$ is  adjacent to $R_1$, then $|R_1|\geq 2$, otherwise, the unique vertex of $R_1$ is a cut vertex of $G$, which contradicts  $\tau(G)=1$.
			Without loss of generality,
			assume that $w_1s\in E(G)$.
			Take $e_1=vw_1$. Let $S_1$ be  a vertex set satisfying Lemma  \ref{1-tough} for  $e_1$.
			Since $vw_1y_1w_2v$ is a cycle of $G$ containing $e_1$, we have   $w_2\in S_1$ or $y_1\in S_1$. Since
			$P\cup uvw_1s$ is another cycle which contains $e_1$ and does not contain $w_2$ and $y_1$, we can obtain  $|S_1|\geq 2$.
			Note that $N_G(w_2)\subseteq\{v,s\}\cup R_2=N_G(w_1)$ and $N_G(y_1)=R_1\subseteq N_G(v)$.
			Thus,
			if $w_2\in S_1$ (or $y_1\in S_1$), then $w_2$ (or $y_1$) is adjacent to exactly one component of $G-S_1$, which contradicts Lemma \ref{two components}.
			Thus, $s$ is adjacent  to $R_2$. Without loss of generality, we may assume that $y_1s\in E(G)$.
			Suppose first that $|R_2|\geq 2$. Then $|R_1|\geq 2$, otherwise, $\omega(G-(\{s\}\cup R_1))\geq 3$, which contradicts   $\tau(G)=1$.
			Take $e_2=w_1y_1$.  Let $S_2$ be a vertex set which  satisfies Lemma  \ref{1-tough} for  $e_2$.
			Since $w_1y_1w_2y_2w_1$ is a cycle of $G$ containing $e_2$, we have  $w_2\in S_2$ or $y_2\in S_2$.
			Moreover, $P\cup uvw_1y_1s$ is another cycle which contains $e_2$ and does not contain $w_2$ and $y_2$. Thus,    $|S_2|\geq 2$.
			Note that $N_G(w_2)=\{v\}\cup R_2=N_G(w_1)$ and $N_G(y_2)\subseteq \{s\}\cup R_1= N_G(y_1)$.
			Therefore,
			if $w_2\in S_2$ (or $y_2\in S_2$), then $w_2$ (or $y_2$) is adjacent to exactly one component of $G-S_2$, which contradicts Lemma \ref{two components}.
			Thus, $|R_2|=1$. Since $\omega(G-(\{v\}\cup R_2))=|R_1|+1$ and $\tau(G)=1$, we  have    $|R_1|=1$. Therefore, $G[L_v]$ is a path of length 2 and $s$ is adjacent to $R_2$.
			
			We finally consider the case $b=3$ and $L_v=\{v\}\cup R_1\cup R_2\cup R_3$.
			Suppose that $s$ is not adjacent to $R_3$. Then $s$ is adjacent to $R_1\cup R_2$.
			By Claim \ref{complete bipartite}, $G[R_1\cup R_2]$ is a complete bipartite graph. Thus, $s$ is adjacent to exactly one of $R_1$ and $R_2$ since $g(uv)=0$.
			Assume first that $s$ is adjacent   to $R_1$.
			Then $|R_1|\geq 2$ and $|R_2|\geq 2$, otherwise, the vertex of $R_1$ or $R_2$ is a cut vertex of $G$, which contradicts   $\tau(G)=1$. Without loss of generality, assume that $w_1s\in E(G)$.  Take $e_3=w_1y_1$.
			Then there exists a vertex set $S_3$ satisfying Lemma \ref{1-tough} for  $e_3$.
			By Claim \ref{complete bipartite},	 $w_1y_1w_2vw_1$ and  $w_1y_1z_1y_2w_1$ are two cycles of $G$ containing $e_3$. Then $|S_3|\geq 2$, and moreover,  $y_2\in S_3$ or $z_1\in S_3$.
			Since
			$N_G(y_2)=R_1\cup R_3=N_G(y_1)$ and $N_G(z_1)= R_2\subseteq N_G(w_1)$, if $y_2\in S_3$ (or $z_1\in S_3$), then $y_2$ (or $z_1$) is adjacent to exactly one component of $G-S_3$, which contradicts Lemma \ref{two components}.
			Suppose now that $s$ is adjacent  only to  $R_2$. Then $|R_2|\geq 2$, otherwise, the vertex of $R_2$ is a cut vertex of $G$, which contradicts  $\tau(G)=1$. Without loss of generality, assume that $y_1s\in E(G)$. Take $e_4=w_1y_1$, and let $S_4$ be  a vertex subset of $G$ satisfying Lemma \ref{1-tough}.
			Since $w_1y_1z_1y_2w_1$ is a cycle of $G$ containing $e_4$, we have   $y_2\in S_4$ or $z_1\in S_4$. Since  $P\cup uvw_1y_1s$ is another cycle which contains $e_4$ and does not contain $y_2$ and $z_1$,  it follows that $|S_4|\geq 2$.
			Note that
			$N_G(y_2)\subseteq \{s\}\cup R_1\cup R_3=N_G(y_1)$ and $N_G(z_1)= R_2\subseteq N_G(w_1)$.
			Therefore,
			if $y_2\in S_4$ (or $z_1\in S_4$), then $y_2$ (or $z_1$) is adjacent to exactly one component of $G-S_4$, which  contradicts   Lemma \ref{two components}.
			
			Therefore, we can obtain   $s$ must be adjacent to $R_3$.
			It follows from $g(uv)=0$ that $s$ is not adjacent to $R_2$.
			We assume first that $|R_3|\geq 2$. Then $|R_2|\geq 2$, otherwise, by Claim \ref{independent}, we have $\omega(G-(\{s\}\cup R_2))\geq 3$, contradicting that $\tau(G)=1$.
			Without loss of generality,
			assume that $z_1s\in E(G)$.
			Take $e_5=y_1z_1$, and let $S_5$ be a vertex subset  of $G$ satisfying
			Lemma \ref{1-tough} for  $e_5$.
			Since $y_1z_1y_2z_2y_1$  is a cycle of $G$ containing $e_5$, we have   $y_2\in S_5$ or $z_2\in S_5$.
			Since $P\cup uvw_1y_1z_1s$ is another cycle which contains $e_5$ and does not contain $y_2$ and $z_2$, it follows that
			$|S_5|\geq 2$.
			Note that $N_G(y_2)= R_1\cup R_3=N_G(y_1)$ and $N_G(z_2)\subseteq\{s\}\cup  R_2= N_G(z_1)$.
			If $y_2\in S_5$ (or $z_2\in S_5$), then $y_2$ (or $z_2$) is adjacent to exactly one component of $G-S_5$, which contradicts Lemma \ref{two components}. Thus, $|R_3|=1$. Suppose that $|R_2|\geq 2$. Then $|R_1|\geq 2$, otherwise, by Claim \ref{independent}, we have $\omega(G-(R_1\cup R_3))\geq 3$,  contradicting that $\tau(G)=1$.
			Take $e_6=w_iy_1$, here $w_i$ is a vertex of $R_1$ such that if $s$ is adjacent to $R_1$, then $w_i$ must be adjacent to $s$.
			Let $S_6$ be a vertex set which satisfies  Lemma \ref{1-tough} for $e_6$.
			Since  $w_iy_1w_jy_2w_i$$(i\neq j)$ is a cycle of $G$ containing $e_6$, we have $w_j\in S_6$ or $y_2\in S_6$.
			Since $P\cup uvw_iy_1z_1s$ is another cycle which contains $e_6$ and does not contain $w_j$ and $y_2$, it follows that $|S_6|\geq 2$.
			If $w_is\in E(G)$, then  $N_G(w_j)\subseteq\{v,s\}\cup R_2=N_G(w_i)$, and if $w_is\notin E(G)$, then  $N_G(w_j)=\{v\}\cup R_2=N_G(w_i)$.
			In either case, $N_G(w_j)\subseteq N_G(w_i)$.
			Thus,
			$w_j$ is adjacent to exactly one component of $G-S_6$ if $w_j\in S_6$, which contradicts Lemma \ref{two components}.
			Therefore, $w_j\notin S_6$. Then   $y_2\in S_6$.
			Since  $N_G(y_2)=  R_1\cup R_3= N_G(y_1)$,  $y_2$ is also adjacent to exactly one component of $G-S_6$, which contradicts Lemma \ref{two components}.
			Thus,  $|R_2|=1$. Since $\omega(G-(\{v,s\}\cup R_2))=1+|R_1|+|R_3|=2+|R_1|$ and $\tau(G)=1$,  we have   $|R_1|=1$. Hence, $G[L_v]$ is a path of length 3 and $s$ is adjacent to $R_3$. Here, $s$ may be adjacent to $R_1$.
			
			Concluding the above discussion, we know that if $vs\notin E(G)$, then $G[L_v]$ is a path and $s$ is adjacent to $R_b$.
			By similar discussion, if $us\notin E(G)$, then $G[L_u]$ is a path and $s$ is adjacent to $L_a$.
		\end{proof}
		
		Since $g(uv)=0$, $s$ is adjacent to at most one
		of $u$ and $v$. Without loss of generality, we assume that $vs\notin E(G)$. Then $b\geq 1$, and thus $a\leq 2$.  By Claim \ref{path},  $G[L_v]$ is a path and $s$ is adjacent to  $R_b$.
		
		If $us\notin E(G)$, then $a\geq 1$, and thus $a+b\geq 2$. Using Claim \ref{path} again, $G[L_u]$ is also a path and
		$s$ is  adjacent to $L_a$.
		Recall that $a+b\leq 3$, then the component  $C(e)$ is a path of length 3 or 4. Since $s$ is adjacent to both $L_a$ and $R_b$,
		we can obtain   $G$ is a  hamiltonian graph.
		By Lemma \ref{hamiltonian}, $G$ is a cycle, and thus $G$ is $C_5$ or $C_6$ in this case.
		
		Suppose in the following discuss that $us\in E(G)$.
		We divide this case into three subcases according to the value of $a$.

		If $a=0$, then	by Claim \ref{path},
		$C(e)$ is a path of length 2, 3 or 4 and $s$ is adjacent to $R_b$. Then $G$ is  a hamiltonian graph, and so    $G$ is $C_4$, $C_5$ or $C_6$ by Lemma \ref{hamiltonian}.

		If $a=1$, then
		since $g(uv)=0$, it follows that $s$ is not adjacent to $L_1$. Hence, $u$ is a cut vertex of $G$, which contradicts   $\tau(G)=1$.

		If $a=2$, then $b=1$. By Claim \ref{path},  $|R_1|=1$.
		Denote  the vertex of $R_1$ by $y_1$. For the edge $e_7=us$, let $S_7$ be a
		vertex subset of $G$ satisfying Lemma \ref{1-tough}.
		By Claim \ref{path},
		$usy_1vu$ is a cycle of $G$ containing $e_7$. Therefore,  $v\in S_7$ or $y_1\in S_7$.
		Since $u$ is not a cut vertex and $g(uv)=0$, $s$ must be adjacent to
		a vertex of $L_2$, say $x_i$. Using Claim \ref{connected},
		we can obtain the subgraph $G[L_u]$ is connected, then the path $usx_i$ together with a path  from $x_i$ to $u$ in $G[L_u]$ consist of anther cycle which contains $e_7$ and does not contain $v$ and $y_1$.
		Thus, $|S_7|\geq 2$.
		Since $N_G(v)=\{u,y_1\}\subseteq N_G(u)\cup N_G(s)$ and $N_G(y_1)=\{v,s\}\subseteq N_G(u)\cup N_G(s)$, if $v\in S_7$ (or $y_1\in S_7$), then $v$ (or $y_1$) is adjacent to exactly one component of $G-S_7$, which contradicts Lemma \ref{two components}.
		
		Therefore, we conclude that if $S$ is not a vertex cut of $G$, then $G$ is $C_4$, $C_5$ or $C_6$.

		\vskip3pt\noindent{\bf Case 2.} $S$ is a  vertex cut of $G$.
		
		In this case, $|S|=\omega(G-S)\geq 2$. Let $S=\{s_1,s_2,\dots,s_k\}$, where $k=|S|\geq 2$.
		Let $V(D(e))$ denote the vertex set of $D(e)$.
		Then any induced path in $C(e)$ is  of length at most 2 and $a+b\leq 1$. According to the symmetric property of $u$ and $v$, we only need to discuss the cases that $a=b=0$ and $a=0,b=1$, that is, $V(C(e))=\{u,v\}$ and $V(C(e))=\{u,v\}\cup R_1$.

		Let $N_S(C(e))$ denote the set of neighbors of all the vertices in $V(C(e))$ within $S$, and let $S'=S-N_S(C(e))$. Here, $S'$ may be an empty set.
		Let $\{A_1,A_2,A_3,A_4\}$ be a partition of $N_S(C(e))$ such that
		$$A_1=\{s\in N_S(C(e))| us \in E(G), vs\notin E(G)\},$$
		$$A_2=\{s\in N_S(C(e))| us \in E(G), vs\in E(G)\},$$
		$$A_3=\{s\in N_S(C(e))| us \notin E(G), vs\in E(G)\},$$
		$$A_4=\{s\in N_S(C(e))| us \notin E(G), vs\notin E(G)\}.$$
		According to the definition above, if $V(C(e))=\{u,v\}$, then $A_4=\emptyset$, and if $V(C(e))=\{u,v\}\cup R_1$ and $A_4\neq \emptyset$, then $A_4$ is only adjacent to $R_1$.
		
		\begin{claim}
			\label{A1A3}
			\rm
			$A_1\neq \emptyset$, and if $V(C(e))=\{u,v\}$, then $A_3\neq \emptyset$.
		\end{claim}
		\begin{proof}[\rm{Proof.}]
			Suppose to the contrary that $A_1=\emptyset$.
			Since $e$ is not a bridge of $G$, it follows that $A_2\neq \emptyset$.
			Then  $N_G(u)=\{v\}\cup A_2\subseteq N_G[v]$.
			Let $s_i$ be a vertex of $A_2$. Take $e_8=vs_i$, and let $S_8$ be a vertex set satisfying  Lemma \ref{1-tough} for $e_8$.
			Since $uvs_iu$ is a triangle of $G$,  we have $u\in S_8$.
			If $|S_8|=1$, then $S_8=\{u\}$ and $e_8$ is a bridge of $G-\{u\}$.
			Let $L_{s_i}$ denote the component of $(G-e_8)-\{u\}$ that contains $s_i$.
			Since $uvs_iu$ is a triangle, according to Lemma \ref{triangle}, we have
			$L_{s_i}-\{s_i\}\neq \emptyset$. Since $s_i$ is not a cut vertex of $G$, $u$ has neighbors in  $L_{s_i}-\{s_i\}$, say $u'$. It follows from $e_8$ is a bridge of $G-\{u\}$ that  $vu'\notin E(G)$, a contradiction with  $N_G(u)\subseteq N_G[v]$.
			If $|S_8|\geq 2$, then since $N_G(u)\subseteq N_G[v]$, $u$ is adjacent to exactly one component of $G-S_8$,  a contradiction obtained by Lemma \ref{two components}. Thus, $A_1\neq \emptyset$.
			If  $V(C(e))=\{u,v\}$, then $A_4=\emptyset$. By the symmetric property of $u$ and $v$, we conclude that $A_3\neq \emptyset$ by similar discussion above.
		\end{proof}

		\begin{claim}
			\label{all the vertices of D(e)}
			\rm
			If $|S|\geq 3$, then each vertex of $A_1\cup A_3$ is adjacent to all the vertices of $D(e)$.
		\end{claim}
		\begin{proof}[\rm{Proof.}]
			Since $|S|\geq 3$, we have $\omega(G-S)\geq 3$.
			Let $s_i$ be an arbitrary vertex of $A_1\cup A_3$.
			By Lemma \ref {two components}, $s_i$ is adjacent to  at least
			one component of $D(e)$, say $D_1$. Let $t_1\in V(D_1)$ such that $s_it_1\in E(G)$. Then $vus_it_1$ or $uvs_it_1$ is an induced subgraph  $P_4$ of $G$. Since $G$ is $(K_1\cup P_4)$-free,
			$s_i$ is adjacent to all the vertices of $D(e)-D_1$. Let $t_2$ be a vertex in $D(e)-D_1$ such that $s_it_2\in E(G)$. Then $vus_it_2$ or $uvs_it_2$ is an induced subgraph $P_4$ of $G$.
			It follows from $G$ is $(K_1\cup P_4)$-free that $s_i$ is adjacent to all the vertices of $D_1$. Thus, $s_i$ is adjacent to all the vertices of $D(e)$. According to the arbitrary property of $s_i$, each vertex of $A_1\cup A_3$ is adjacent to all the vertices of $D(e)$.
		\end{proof}
		
		\begin{claim}
			\rm
			\label{A4}
			If $A_4\neq \emptyset$, then each vertex of $A_4$ is adjacent to all the vertices of $V(D(e))\cup S'$.
		\end{claim}
		\begin{proof}[\rm{Proof.}]
			By the definition of $A_4$, if $A_4\neq \emptyset$, then $R_1\neq \emptyset$, and each vertex of $A_4$ has at least one neighbor in $R_1$. Let $s_i$ be an arbitrary vertex of $A_4$, and let $y_j$ be a neighbor of $s_i$ in $R_1$. Then $uvy_js_i$ is an induced subgraph $P_4$ of $G$. Since $G$ is $(K_1\cup P_4)$-free, $s_i$ is adjacent to all the vertices of $V(D(e))\cup S'$. By the arbitrary property of $s_i$, we can obtain each vertex of $A_4$ is adjacent to all the vertices of $V(D(e))\cup S'$.
		\end{proof}
		
		\vskip3pt\noindent{\bf Case 2.1.} $V(C(e))=\{u,v\}$.
		
		In this case, $A_4=\emptyset$.
		By Claim \ref{A1A3},  we have $A_1\neq \emptyset$ and $A_3\neq\emptyset$.
		The following discussions are based on whether $A_2$ is an empty set.

		Assume first that 	$A_2\neq \emptyset$. Then $|S|\geq 3$.
		Let $s_i$ be a vertex of $A_i$ for each $i=1,2,3$.
		By Lemma \ref{two components}, $s_2$ has at least one neighbor in $D(e)$.
		Let $t_1,t_2$ be two vertices of $D(e)$ such that $s_2t_1\in E(G)$.
		By Claim \ref{all the vertices of D(e)},
		$s_1$ and $s_3$ are adjacent to both $t_1$ and $t_2$.
		Take $e_9=us_1$.
		By  Lemma \ref{1-tough}, there exists a vertex subset $S_9$ of $G$ such that
		\begin{center}
				$\omega(G - S_9) = |S_9|$ and $\omega((G - e_9) - S_9) = |S_9| + 1.$
		\end{center}
		Since $us_1t_1s_2u$ and $us_1t_2s_3vu$ are two cycles of $G$ containing $e_9$,
		we have   $|S_9|\geq 2$. Then $\omega(G-S_9)\geq 2$, and thus $V(D(e_9))\neq \emptyset$.
		Since $s_1$ is adjacent to all the vertices of $D(e)$,
		it follows that $V(D(e))\subseteq S_9\cup V(C(e_9))$.
		Since $u$ is adjacent to  all the vertices of $\{v\}\cup A_1\cup A_2$, we can obtain   $\{v\}\cup A_1\cup A_2\subseteq S_9\cup V(C(e_9))$.
		Thus, $V(D(e_9))\subseteq A_3\cup S'$.
		
		If $A_3\cap V(D(e_9))\neq \emptyset$, then since $\{v\}\cup V(D(e))\subseteq S_9\cup V(C(e_9))$ and each vertex of  $A_3$ is adjacent to all the vertices of $\{v\}\cup V(D(e))$,
		it follows that $\{v\}\cup V(D(e))\subseteq S_9$, and thus $$|S_9|\geq |V(D(e))|+1\geq |D(e)|+1=\omega(G-S)-1+1=|S|.$$
		Since $V(D(e_9))\subseteq A_3\cup S'$, we have  
		$$\omega(G-S_9)\leq |A_3\cup S'|+1=|A_3|+|S'|+1=|S|-|A_1|-|A_2|+1\leq |S|-1,$$
		which contradicts  $\omega(G-S_9)=|S_9|.$

		Therefore, $A_3\cap V(D(e_9))=\emptyset$, and so $V(D(e_9))\subseteq S'$.
		Let $s'\in S'\cap V(D(e_9))$.
		Since $s_1\in V(C(e_9))$, we have $s_1s'\notin E(G)$.
		By Lemma \ref{two components}, $s'$ is adjacent to at least two components of $D(e)$. Let $t_i$ be a neighbor of $s'$ in $D_i$, where $D_i$ is a component in $D(e)$.
		Since $|S|\geq 3$, by Claim \ref{all the vertices of D(e)},
		$s_1$ is adjacent to all the vertices of $D(e)$.
		Then $s'$ is adjacent to all the vertices of $D(e)-D_i$.
		Otherwise, let $t_i'$ be a vertex of $D(e)-D_i$ such that $t_i's'\notin E(G)$, then the subgraph induced by $\{v,s',t_i,s_1,t_i'\}$ is a $K_1\cup P_4$ of $G$, a contradiction.
		By similar discussion, we can show that
		$s'$ is adjacent to all the vertices of $D_i$.
		Otherwise, let
		$t_j$ be a vertex of $D(e)-D_i$ and let $t_j'$ be a vertex of $D_i$ such that $s't_j'\notin E(G)$, then
		the subgraph induced by $\{v,t_j',s_1,t_j,s'\}$ is
		a $K_1\cup P_4$ of $G$, a contradiction.
		Thus, $s'$ is adjacent to  all the vertices of $D(e)$.
		By the arbitrary property of $s'$,
		each vertex of $S'\cap V(D(e_9))$ is adjacent to all the vertices of $D(e)$. Note that $V(D(e))\subseteq S_9\cup V(C(e_9))$.
		Then	$V(D(e))\subseteq S_9$.
		Since $G$ is $(K_1\cup P_4)$-free and $s_1s'\notin E(G)$, we can obtain  $s_1s_3\in E(G)$ or $s's_3\in E(G)$.
		Otherwise, the subgraph induced by $\{s', s_1,u,v,s_3\}$ is
		a $K_1\cup P_4$ of $G$, a contradiction.
		If $s_1s_3\in E(G)$, then since
		$us_1s_3vu$ is a cycle,
		we have $v\in S_9$ or $s_3\in S_9$.
		If $s_3s'\in E(G)$, then it follows from   $uvs_3s'$ is a path from
		$C(e_9)$ to $D(e_9)$ that
		at least one vertex of $\{v,s_3\}$ is contained in $S_9$.
		Therefore, in all the cases discussed as above, we have   $v\in S_9$ or $s_3\in S_9$. Recalling that $V(D(e))\subseteq S_9$, we can obtain  
		$$|S_9|\geq |V(D(e))|+1\geq |D(e)|+1=\omega(G-S)-1+1=|S|.$$
		Since $V(D(e_9))\subseteq S'$,  we can obtain  
		$$\omega(G-S_9)\leq |S'|+1=|S|-|A_1|-|A_2|-|A_3|+1\leq |S|-3+1=|S|-2, $$
		a contradiction with $\omega(G-S_9)=|S_9|$.

		Suppose now that	$A_2=\emptyset$. This implies that $g(uv)=0$. In this case, 	the edge $e=uv$ satisfies the condition $f(uv)=$ max$\{f(xy):xy \in E(G)\}$.
		Note that $N_G(u)\cup N_G(v)=\{u,v\}\cup A_1\cup A_3$. Therefore,
		$$f(uv)=|N_G(u)\cup N_G(v)|=|A_1|+|A_3|+2.$$
		Assume that $|S|\geq 3$. By Claim \ref{all the vertices of D(e)}, each vertex of $A_1\cup A_3$ is adjacent to all the vertices
		of $D(e)$.
		Let $s_1$ be a vertex of $A_1$.
		Then $\{u,v\}\cup A_1\cup V(D(e))\subseteq N_G(u)\cup N_G(s_1)$.
		Without loss of generality, we may assume that $|A_1|\geq |A_3|$.
		Thus,
		$$f(us_1)=|N_G(u)\cup N_G(s_1)|
		\geq 2+|A_1|+|V(D(e))|\geq 2+|A_1|+|D(e)|=2+|A_1|+\omega(G-S)-1=|A_1|+|S|+1.$$
		Combining  this with the inequality $f(uv)\geq f(us_1)$,  we have  
		$$|A_1|+|S|+1\leq |A_1|+|A_3|+2.$$ Thus,
		\begin{equation}
			|S|\leq |A_3|+1.
		\end{equation}
		Note that $|S|\geq |A_1|+|A_3|$.
		By (1), we can obtain   $|A_1|\leq 1$, and so $|A_1|=1$. Since $|A_1|\geq |A_3|$ and $A_3\neq \emptyset$, we have  $|A_3|=1$. Applying (1) again, we can obtain   $|S|\leq 2$, a contradiction with the assumption that $|S|\geq 3$.
		
		Therefore, $|S|=2$, and thus $\omega(G-S)=2$.
		Note that $A_1\neq \emptyset$ and $A_3\neq \emptyset$. Then
		$S=A_1\cup A_3$.
		Let $A_1=\{s_1\}$ and $A_3=\{s_2\}$.
		Since $N_G(u)\cup N_G(v)=\{u,v,s_1,s_2\}$, we have  $$f(uv)=|N_G(u)\cup N_G(v)|=4.$$
		By Lemma \ref{two components},  $N_{V(D(e))}(s_1)\neq \emptyset$ and $N_{V(D(e))}(s_2)\neq \emptyset$.
		Since $N_G(u)=\{s_1,v\}$ and $\{u\}\cup N_{V(D(e))}(s_1)\subseteq N_G(s_1)$,
		we have  
		$$f(us_1)=|N_G(u)\cup N_G(s_1)|\geq 3+|N_{V(D(e))}(s_1)|.$$
		By the maximality  of	$f(uv)$, it follows that
		$f(us_1)=4$ and $|N_{V(D(e))}(s_1)|=1$. This implies that $s_1s_2\notin E(G)$.
		Similarly, $|N_{V(D(e))}(s_2)|=1$.
		Since $s_1uvs_2$ is an induced subgraph $P_4$ of $G$, each vertex of $D(e)$ is adjacent to at least one of $s_1$ and $s_2$, and thus $|V(D(e))|\leq 2$.
		If $|V(D(e))|=1$, then $G$ is $C_5$, and if $|V(D(e))|=2$, then $G$ is $C_6$.

		\vskip3pt\noindent{\bf Case 2.2.} $V(C(e))=\{u,v\}\cup R_1$.
		
		Denote the vertex set $R_1=\{y_1,y_2,\dots, y_m\}$, here $m=|R_1|\geq 1$.
		Since $v$ is not a cut vertex of $G$, $N_{N_S(C(e))}(R_1)\neq \emptyset$.
		Without loss of generality, let
		$y_1\in R_1$ such that $N_{N_S(C(e))}(y_1)\neq \emptyset$.
		Let $A_1^*$ denote a subset of $A_1$ such that
		$N_{R_1}(A_1^*)\neq R_1$.
		\begin{claim}
			\rm
			\label{A1v1}
			Each vertex of $A_1^*$ is adjacent to all the vertices of $V(D(e))\cup S'$.
		\end{claim}
		\begin{proof}[\rm{Proof.}]
			Let $s_i$ be an arbitrary vertex of $A_1^*$, and let $y_j$ be a vertex of $R_1$ such that  $s_iy_j\notin E(G)$.
			Then $s_iuvy_j$ is an induced subgraph $P_4$ of $G$. Since $G$ is $(K_1\cup P_4)$-free, $s_i$ is adjacent to all the vertices of $V(D(e))\cup S'$. By the arbitrary property of $s_i$, we can obtain   each vertex of $A_1^*$ is adjacent to all the vertices of $V(D(e))\cup S'$.
		\end{proof}

		\vskip3pt\noindent{\bf Case 2.2.1.} $|R_1|\geq 2$.
		
		By Claim \ref{connected}, $y_1$ has at least one neighbor in $R_1$.
		Without loss of generality,
		we	assume that $y_2$ is adjacent to $y_1$.
		Take $e_{10}=vy_1$.
		By  Lemma \ref{1-tough}, there exists a vertex subset $S_{10}$ of $G$ such that
		\begin{center}
			$\omega(G - S_{10}) = |S_{10}|$ and $\omega((G - e_{10}) - S_{10}) = |S_{10}| + 1.$
		\end{center}
		Since $vy_1y_2v$ is a triangle of $G$, $y_2\in S_{10}$.
		We first show   $|S_{10}|\geq 2$.
		Recall that $N_{N_S(C(e))}(y_1)\neq \emptyset$.
		If $y_1$ is adjacent to a vertex of $A_1\cup A_2\cup A_3$,
		say $s_i$, then $vus_iy_1v$ or $vs_iy_1v$ is a
		cycle which contains $e_{10}$ and does not contain $y_2$, and thus $|S_{10}|\geq 2$.
		If $y_1$ is  not adjacent to any vertex of $A_1\cup A_2\cup A_3$, then $y_1$ must be  adjacent to a vertex of $A_4$, say $s_j$. In this case, $A_1^*=A_1$.
		By Claim \ref{A1v1}, each vertex of $A_1$ is adjacent to all the vertices of $D(e)$.
		By Claim \ref{A4}, $s_j$ is  also adjacent to all the vertices of $D(e)$.
		Let $s_l\in A_1$ and $t_1\in V(D(e))$.
		Then $s_lt_1,s_jt_1\in E(G)$ and
		$vy_1s_jt_1s_luv$ is a cycle which contains $e_{10}$ and does not contain $y_2$, and  then  $|S_{10}|\geq 2$.
		Thus, in all the cases as above, $|S_{10}|\geq 2$.
		By Lemma \ref{two components}, $N_{V(D(e_{10}))}(y_2)\neq \emptyset$.
		Since $N_G(v)=A_2\cup A_3\cup R_1\cup \{u\}$ and $N_G(y_2)\subseteq A_1\cup A_2\cup A_3\cup A_4\cup R_1\cup \{v\}$,
		we have  
		$N_{V(D(e_{10}))}(y_2)\subseteq A_1\cup A_4$ and $V(D(e_{10}))\subseteq A_1\cup A_4\cup S'\cup V(D(e))$.
		
		If $N_{V(D(e_{10}))}(y_2)\cap  A_1\neq \emptyset$,
		then $u\in S_{10}$. Let $A_1'=V(D(e_{10}))\cap  A_1$.
		Obviously, $A_1'\neq \emptyset$. Furthermore,
		each vertex of $A_1'$ is not adjacent to $y_1$.
		Then $A_1'\subseteq A_1^*$.
		By Claim \ref{A1v1},
		each vertex of $A_1'$ is adjacent to all the vertices of $V(D(e))\cup S'$.
		First suppose that $(V(D(e))\cup S')\cap V(D(e_{10}))=\emptyset$. Then $V(D(e))\cup S'\subseteq S_{10}$.
		Reminding that $\{u,y_2\}\subseteq S_{10}$, then $$|S_{10}|\geq |S'|+|V(D(e))|+2\geq |S'|+|D(e)|+2=|S'|+\omega(G-S)-1+2=|S'|+|S|+1\geq |S|+1.$$
		Let $A_4'=V(D(e_{10}))\cap  A_4$.
		Since $N_{N_S(C(e))}(y_1)\subseteq V(C(e_{10}))\cup S_{10}$,
		we get that $N_{N_S(C(e))}(y_1)\cap (A_1'\cup A_4')=\emptyset$.
		Since  $N_{N_S(C(e))}(y_1)\neq \emptyset$,
		we have  
		$$|S|\geq |A_1'|+|A_4'|+1.$$
		According to the above discussion, we can obtain   $V(D(e_{10}))=A_1'\cup A_4'$.
		Thus, $$\omega(G-S_{10})=|D(e_{10})|+1\leq |V(D(e_{10}))|+1
		=|A_1'|+|A_4'|+1\leq |S|-1+1=|S|,$$
		contradicting that $\omega(G - S_{10}) = |S_{10}|$.
		Now we suppose that $(V(D(e))\cup S')\cap  V(D(e_{10}))\neq \emptyset$. By Claims \ref{A4} and  \ref{A1v1},  each vertex of $A_1'\cup A_4'$ is adjacent to all the vertices of $V(D(e))\cup S'$. Thus, 
		the vertices of $A_1'\cup A_4'$ together with the vertices of $(V(D(e))\cup S')\cap  V(D(e_{10}))$ are contained in exactly one component of $D(e_{10})$. Then
		$\omega(G-S_{10})=2$, and so $|S_{10}|=2$, which means that  $S_{10}=\{u,y_2\}$. Then $N_{N_S(C(e))}(y_1)\subseteq V(C(e_{10}))$.  By Claims \ref{A4} and \ref{A1v1}, $V(D(e))\subseteq S_{10}\cup V(D(e_{10}))$.
		Therefore, $V(D(e))\subseteq V(D(e_{10}))$ since $S_{10}=\{u,y_2\}$.
		However, each vertex of $N_{N_S(C(e))}(y_1)$ is adjacent to at least one vertex of $D(e)$, which means that $V(C(e_{10}))$ is adjacent to $V(D(e_{10}))$, a contradiction.

		If $N_{V(D(e_{10}))}(y_2)\cap  A_4\neq \emptyset$, then let $s_i$ be a vertex of $N_{D(e_{10})}(y_2)\cap  A_4$. Then $y_1s_i\notin E(G)$ since $y_1\in V(C(e_{10}))$. By Claim \ref{A4}, $s_i$ is adjacent to all the vertices of $D(e)$. Let $t_j$ be an arbitrary vertex of $D(e)$.
		By the definition of $A_4$, $us_i\notin E(G)$, and then
		the subgraph induced by $\{u,y_1,y_2,s_i,t_j\}$ is a
		$K_1\cup P_4$ in $G$, a contradiction.
		
		\vskip3pt\noindent{\bf Case 2.2.2.} $|R_1|=1$.
		
		Denote
		$R_1=\{y_1\}$.
	By the definition of $A_4$, we can obtain  $A_4\subseteq N_G(y_1)$.
		We first list some claims.
		
		\begin{claim}
			\rm
			\label{v1 A1}
			$y_1$ is not adjacent to $A_1$.
		\end{claim}
		\begin{proof}[\rm{Proof.}]
			Suppose to the contrary that $y_1$ is  adjacent to $A_1$.
			Let $s_i$ be  a vertex of $A_1$ such that $s_iy_1\in E(G)$.
			We first show that $(A_1-\{s_i\})\cup A_2\neq \emptyset$. If not, then  
			 $A_1=\{s_i\}$ and $A_2=\emptyset$.  Therefore, $g(uv)=0$ and
			$f(uv)=$ max$\{f(xy):xy \in E(G)\}$.
			It is claimed  that $A_4\neq \emptyset$. Otherwise, if $A_4=\emptyset$, then to avoid $s_i$ being a cut vertex, there is
			$A_3\neq \emptyset$. Since $g(uv)=0$, $y_1$ is not adjacent to $A_3$. Then
			$N_G(y_1)=\{v,s_i\}$.
			Note that $N_G(u)=\{v,s_i\}$. Then
			$\omega(G-\{v,s_i\})\geq 3$ since $\omega(G-S)\geq 2$,
			which contradicts  $\tau(G)=1$.
			Therefore, $A_4\neq \emptyset$.
			Since $N_G(v)=\{u,y_1\}\cup A_3$ and $N_G(y_1)=\{v,s_i\}\cup A_4$, we have   $$f(uv)=|N_G(u)\cup N_G(v)|=|\{u,v,y_1,s_i\}\cup A_3|=4+|A_3|$$ and
			$$f(vy_1)=|N_G(v)\cup N_G(y_1)|=4+|A_3|+|A_4|\geq 5+|A_3|,$$
			a contradiction to the maximality of $f(uv)$.
			Thus,
			$(A_1-\{s_i\})\cup A_2\neq \emptyset$.
			
			Take $e_{11}=us_i$. 	By  Lemma \ref{1-tough}, there exists a vertex subset $S_{11}$ of $G$ such that
			\begin{center}
				$\omega(G - S_{11}) = |S_{11}|$ and $\omega((G - e_{11}) - S_{11}) = |S_{11}| + 1.$
			\end{center}
			Since $us_iy_1vu$ is a cycle of $G$, we have   $v\in S_{11}$ or $y_1\in S_{11}$.
			Recall that $(A_1-\{s_i\})\cup A_2\neq \emptyset$. Let $s_j$ be a vertex of $(A_1-\{s_i\})\cup A_2$.
			If $|S|=2$, that is, $S=\{s_i,s_j\}$,
			then $|D(e)|=1$.
			By Lemma \ref{two components},
			$N_{V(D(e))}(s_i)\neq \emptyset$ and $N_{V(D(e))}(s_j)\neq \emptyset$. Let $t_1\in N_{V(D(e))}(s_i)$ and $t_2\in N_{V(D(e))}(s_j)$. Here, $t_1$ and $t_2$ may be the same vertex.
			Since $D(e)$ is a connected component, a path from $t_1$ to $t_2$ in $D(e)$, together with the path $t_1s_ius_jt_2$ consist of a cycle which contains $e_{11}$ and  does not contain $v$ and $y_1$. Thus, $|S_{11}|\geq 2$.
			If $|S|\geq 3$,  then by Lemma \ref{two components},
			let $t_3$ be a vertex of $D(e)$ such that
			$s_jt_3\in E(G)$.
			By Claim \ref{all the vertices of D(e)}, $s_it_3\in E(G)$. Then $us_it_3s_ju$ is also a cycle which contains $e_{11}$ and does not contain $v$ and $y_1$.
			Hence, $|S_{11}|\geq 2$.
			Therefore, in all the cases discussed above, $|S_{11}|\geq 2$.

			By Lemma \ref{two components}, $N_{V(D(e_{11}))}(v)\neq \emptyset$ if $v\in S_{11}$ and $N_{V(D(e_{11}))}(y_1)\neq \emptyset$ if $y_1\in S_{11}$.
			Since $N_G(v)\subseteq \{u,y_1\}\cup N_S(C(e))$ and $N_G(y_1)\subseteq \{v\}\cup N_S(C(e))$, it follows that
			$N_S(C(e))\cap V(D(e_{11}))\neq \emptyset$.
			Since $u$ is adjacent to all the vertices of $A_1\cup A_2$, we have   $(A_3\cup A_4)\cap V(D(e_{11}))\neq \emptyset$.
			This also implies that $A_3\cup A_4\neq \emptyset$.
			Combining this with $(A_1-\{s_i\})\cup A_2\neq \emptyset$, we know   $|S|\geq 3$.
			By Claims \ref{all the vertices of D(e)} and  \ref{A4},
			each vertex of $A_1\cup A_3\cup A_4$ is adjacent to all the vertices of $D(e)$.
			Since $s_i\in A_1\cap V(C(e_{11}))$ and $(A_3\cup A_4)\cap
			V(D(e_{11}))\neq \emptyset$,
			we have   
			$V(D(e))\subseteq S_{11}$. Reminding that $v\in S_{11}$ or $y_1\in S_{11}$, then
			$$|S_{11}|\geq |V(D(e))|+1\geq |D(e)|+1=\omega(G-S)-1+1=|S|.$$
			Since $V(D(e_{11}))\subseteq A_3\cup A_4\cup S'$ and $(A_1-\{s_i\})\cup A_2\neq \emptyset$, we can obtain 
			$$\omega(G-S_{11})\leq |A_3|+|A_4|+|S'|+1=|S|-|A_1|-|A_2|+1\leq |S|-2+1=|S|-1,$$
			which contradicts   $|S_{11}|=\omega(G-S_{11})$.
			Therefore, $y_1$ is not adjacent to $A_1$.
		\end{proof}
		
		By Claims \ref{A1v1} and   \ref{v1 A1}, the following claim holds obviously.
		
		\begin{claim}
			\rm
			\label{A1V(D(e))S'}
			Each vertex of $A_1$ is adjacent to all the vertices of $V(D(e))\cup S'$.
		\end{claim}

		\begin{claim}
			\label{A3A4}
			\rm
			$y_1$ is adjacent to $A_3\cup A_4$ and $A_3\cup A_4\neq \emptyset$.
		\end{claim}
		\begin{proof}[\rm{Proof.}]
			Suppose to the contrary that $y_1$ is not adjacent to $A_3\cup A_4$.
			Note that $v$ is not a cut vertex. This implies that $N_S(y_1)\neq \emptyset$.
			By Claim \ref{v1 A1}, we have $A_2\neq \emptyset$ and $y_1$ is adjacent to $A_2$.
			Let $s_i\in A_2$ such that $y_1s_i\in E(G)$.
			Take $e_{12}=vs_i$. Let
			$S_{12}$ be a  vertex set satisfying Lemma \ref{1-tough} for $e_{12}$.
			Since $uvs_iu$ and $vs_iy_1v$ are two triangles of $G$ containing the edge $e_{12}$, we have   $u\in S_{12}$ and $y_1\in S_{12}$, and thus $|S_{12}|\geq 2$.
			However, since $N_G(y_1)\subseteq\{v\}\cup A_2\subseteq N_G(v)\cup N_G(s_i)$, we can obtain  $y_1$ is adjacent to exactly one component of $G-S_{12}$, which contradicts Lemma \ref{two components}. Thus, $y_1$ is adjacent to $A_3\cup A_4$ and $A_3\cup A_4\neq \emptyset$.
		\end{proof}

		Suppose first that $A_2\neq \emptyset$.
		By Claims \ref{A1A3} and   \ref*{A3A4}, we have   $A_1\neq \emptyset$ and $A_3\cup A_4\neq \emptyset$,
		and so $|S|\geq 3$.
		By Claims \ref{all the vertices of D(e)} and   \ref{A4}, each vertex of $A_1\cup A_3\cup A_4$ is adjacent to all the vertices of $D(e)$.
		Let $s_1$ and $s_2$ be two vertices of $A_1$ and $A_2$, respectively.
		Take $e_{13}=us_1$.
		By  Lemma \ref{1-tough}, there exists a vertex subset $S_{13}$ of $G$ such that
		\begin{center}
			$\omega(G - S_{13}) = |S_{13}|$ and $\omega((G - e_{13})- S_{13}) = |S_{13}| + 1.$
		\end{center}
		By Lemma \ref{two components},  $N_{V(D(e))}(s_2)\neq \emptyset$. Let  $t_1\in N_{V(D(e))}(s_2)$.
		Then $s_1t_1\in E(G)$, and so  $us_1t_1s_2u$ is a cycle of $G$ containing the edge $e_{13}$.
		Let $s_3$ be a vertex of $A_3\cup A_4$.
		Note that $\omega(G-S)=|S|\geq 3$.
		Then there must be a vertex, denoted by $t_2$, which is different from $t_1$ in $V(D(e))$. Thus,
		$s_3t_2,s_1t_2\in E(G)$, and so $us_1t_2s_3vu$ or $us_1t_2s_3y_1vu$ is another cycle which contains $e_{12}$ and does not contain $s_2$ and $z_1$. Therefore,  $|S_{13}|\geq 2$, and so
		$V(D(e_{13}))\neq \emptyset$.
		By Claim \ref{A1V(D(e))S'},	$s_1$ is adjacent to
		all the vertices of $V(D(e))\cup S'$. Then $V(D(e))\cup S'\subseteq V(C(e_{13}))\cup S_{13}$.
		Since $u$ is adjacent to all the vertices of $\{v\}\cup A_1\cup A_2$,  we have   $\{v\}\cup A_1\cup A_2\subseteq V(C(e_{13}))\cup S_{13}$.
		Thus, $V(D(e_{13}))\subseteq \{y_1\}\cup A_3\cup A_4$.
		
		If $A_4\cap V(D(e_{13}))\neq \emptyset$, then by Claim \ref{A4},  each vertex of $A_4$ is adjacent to all the vertices of $V(D(e))\cup S'$,
		it follows that $V(D(e))\cup S'\subseteq S_{13}$.
		Since $uvy_1$ is  a path and $y_1$ is adjacent to each vertex of $A_4$, we have  
		$v\in S_{13}$ or $y_1\in S_{13}$.
		Thus, $$|S_{13}|\geq |V(D(e))|+|S'|+1\geq |D(e)|+|S'|+1=\omega(G-S)-1+|S'|+1=|S|+|S'|\geq |S|.$$
		Note that $V(D(e_{13}))\subseteq \{y_1\}\cup A_3\cup A_4$. If $y_1\in V(D(e_{13}))$, then since $y_1$ is adjacent to all the vertices of $A_4$, it follows that $y_1$ and $A_4\cap V(D(e_{13}))$ are contained in the same component of $D(e_{13})$.
		Therefore, $$\omega(G-S_{13})\leq |A_3|+1+1=|S|-|A_1|-|A_2|-|A_4|-|S'|+2\leq |S|-3-|S'|+2=|S|-1-|S'|\leq |S|-1,$$
		which contradicts   $\omega(G - S_{13}) = |S_{13}|$.
		If $y_1\notin V(D(e_{13}))$, then $V(D(e_{13}))\subseteq  A_3\cup A_4$, and so
		$$\omega(G-S_{13})\leq |A_3|+|A_4|+1=|S|-|A_1|-|A_2|-|S'|+1\leq |S|-2-|S'|+1=|S|-1-|S'|\leq |S|-1,$$
		which contradicts   $\omega(G - S_{13}) = |S_{13}|$.
		Thus, $A_4\cap V(D(e_{13}))=\emptyset$, and so
		$V(D(e_{13}))\subseteq \{y_1\}\cup A_3$.
		
		If $A_3\cap V(D(e_{13}))\neq \emptyset$,
		then since $v$ is adjacent to both $u$ and all the vertices in $A_3$,
	it follows that $v\in S_{13}$.
		By Claim \ref{all the vertices of D(e)}, each vertex of $A_3$ is adjacent to all the vertices of $D(e)$. Therefore,
		$V(D(e))\subseteq  S_{13}$. Thus, $$|S_{13}|\geq |V(D(e))|+1\geq |D(e)|+1=\omega(G-S)-1+1=|S|.$$
		If $y_1\in V(D(e_{13}))$ and $A_4\neq \emptyset$, then
		$$\omega(G-S_{13})\leq |A_3|+1+1=|S|-|A_1|-|A_2|-|A_4|-|S'|+2\leq |S|-3-|S'|+2=|S|-1-|S'|\leq |S|-1.$$
		If $y_1\in V(D(e_{13}))$ and $A_4= \emptyset$, recall that  $y_1$ is adjacent to $A_3\cup A_4$, then $N_{A_3}(y_1)\neq \emptyset$ and
		the vertex of $N_{A_3}(y_1)$ is  either in the same component of $D(e_{13})$ with $y_1$, or contained in $S_{13}$.
		Then
		$$\omega(G-S_{13})\leq (|A_3|-1)+1+1=|A_3|+1=
		|S|-|A_1|-|A_2|-|S'|+1\leq |S|-2-|S'|+1=|S|-|S'|-1
		\leq  |S|-1.$$
		If  $y_1\notin V(D(e_{13}))$, then $V(D(e_{13}))\subseteq  A_3$.
		Therefore, $$\omega(G-S_{13})\leq |A_3|+1=|S|-|A_1|-|A_2|-|A_4|-|S'|+1\leq |S|-2-|A_4|-|S'|+1=|S|-1-|A_4|-|S'|\leq |S|-1.$$
		In each case above, $\omega(G - S_{13}) < |S_{13}|$,
		which contradicts   $\omega(G - S_{13}) = |S_{13}|$.
		
		Thus, $V(D(e_{13}))\cap (A_3\cup A_4)=\emptyset$, and so $V(D(e_{13}))=\{y_1\}$. Then $\omega(G-S_{13})=2$, and thus $|S_{13}|=2$.  It follows from $vy_1\in E(G)$ and $uv\in E(G)$ that
		$v\in S_{13}$.
		Since $us_1t_1s_2u$ is a cycle of $G$ containing $e_{13}$, we have   $t_1\in S_{13}$ or $s_2\in S_{13}$. Thus, $S_{13}=\{v,s_2\}$ or $S_{13}=\{v,t_1\}$.
		This also implies that $S_{13}\cap (A_3\cup A_4)= \emptyset$.
		By Claim \ref{A3A4}, $y_1$ is adjacent to $A_3\cup A_4$.
		Since $V(D(e_{13}))=\{y_1\}$, we have   $N_{A_3\cup A_4}(y_1)\subseteq V(D(e_{13}))$, which contradicts 
		$V(D(e_{13}))\cap (A_3\cup A_4)=\emptyset$.

		Assume now that $A_2=\emptyset$. Then $g(uv)=0$, and so $f(uv)=$ max$\{f(xy):xy \in E(G)\}$ in this case.
		Thus,  $y_1$ is not adjacent to $A_3$.
		Since $\tau(G)=1$, $y_1$ is adjacent to $S$.
		By Claim \ref{v1 A1}, $y_1$ is not adjacent to $A_1$.
		Thus, $A_4\neq \emptyset$.
		Suppose that $|S|\geq 3$.
		Note that $N_G(u)\cup N_G(v)=\{u,v,y_1\}\cup A_1\cup A_3$. Therefore, $$f(uv)=|N_G(u)\cup N_G(v)|=3+|A_1|+|A_3|.$$
		By Claim \ref{A1A3}, $A_1\neq \emptyset$.	Let $s_1$ be a vertex of $A_1$.
		By Claim \ref{all the vertices of D(e)},  $s_1$ is adjacent to all the vertices of $D(e)$.
		Then $\{u,v\}\cup A_1\cup V(D(e))\subseteq N_G(u)\cup N_G(s_1)$, we can obtain   $$f(us_1)= |N_G(u)\cup N_G(s_1)|\geq 2+|A_1|+|V(D(e))|\geq 2+|A_1|+|D(e)|= 2+|A_1|+\omega(G-S)-1=|S|+1+|A_1|.$$
		Since 	$f(uv)\geq f(us_1)$, we can obtain   $|S|+1+|A_1|\leq 3+|A_1|+|A_3|$, and thus
		$$	|A_1|+|A_3|+|A_4|+|S'|=	|S|\leq |A_3|+2.$$
		Therefore, $|A_1|+|A_4|+|S'|\leq 2$. Since $A_1\neq \emptyset$ and $A_4\neq \emptyset$, we have   $|A_1|=1$, $|A_4|=1$ and $|S'|=0$.
		Since $|S|\geq 3$, it follows that $|A_3|\geq 1$.
		Let $s_2$ be a vertex of $A_3$.
		By Claim \ref{all the vertices of D(e)}, $s_2$ is adjacent to all the vertices of $D(e)$. Then
		$\{u,v,y_1\}\cup A_3\cup V(D(e))\subseteq N_G(v)\cup N_G(s_2)$, and thus $$f(vs_2)=|N_G(v)\cup N_G(s_2)|\geq 3+|A_3|+|V(D(e))|\geq 3+|A_3|+|D(e)|=3+|A_3|+\omega(G-S)-1=|S|+2+|A_3|.$$
		Since $f(uv)\geq f(vs_2)$, we have   $|S|+2+|A_3|\leq 3+|A_1|+|A_3|$.
		Thus,  $|S|\leq |A_1|+1=2$, a contradiction with the assumption that $|S|\geq 3$.
		
		Thus, $|S|=2$ and $|D(e)|=1$. Since $A_1\neq \emptyset$ and $A_4\neq \emptyset$, we can obtain   $S=A_1\cup A_4$.
		By Claim \ref{A1V(D(e))S'},
		the unique vertex of $A_1$ is adjacent to all the vertices of $D(e)$.
		Since  $g(uv)=0$,  we have  $|V(D(e))|=1$.
		By Lemma \ref{two components}, the unique vertex of $A_4$ is also adjacent to the vertex of $D(e)$. Thus, $G$ is a hamiltonian graph. By Lemma \ref{hamiltonian}, $G$ is $C_6$. This completes the proof of Theorem \ref{Th1}.
	\end{proof}

\end{document}